\documentclass[a4paper,12pt]{article}
\usepackage{array,amsfonts,amsmath,graphicx,mathrsfs,amssymb,amsthm,latexsym}
\usepackage[latin1]{inputenc}
\usepackage{color}

\usepackage{csquotes}

\newtheorem{thm}{Theorem}[section]

\newtheorem{lemma}[thm]{Lemma}

\def \cH {{\cal H}}

\begin{document}
\title{Uniformly resolvable decompositions of $K_v$ into one $1$-factor and $n$-stars when $n>1$ is odd}

\author
 {Jehyun Lee and  Melissa Keranen\\
\small Department of Mathematical Sciences \\
\small Michigan Technological University\\
}

\maketitle

\vspace{5 mm}

\begin{abstract}
We consider uniformly resolvable decompositions of $K_v$ into subgraphs such that each resolution class contains only blocks isomorphic to the same graph. We give a complete solution for the case in which one resolution class is $K_2$ and the rest are $K_{1,n}$ where $n>1$ is odd.

\end{abstract}

\section{Introduction}\label{intro}

Let $G$ be a graph with vertex set $V(G)$ and edge set ${E }(G)$. An $\cH$-$decomposition$ of the graph $G$ is a collection of edge disjoint subgraphs $\cH = \{H_1,H_2, \dots, H_a \}$ such that every edge of $G$ appears in exactly one graph $H_i \in \cH$. 

The subgraphs, $\cH_i \in \cH$, are called blocks. An $\cH$-decomposition is called $resolvable$ if the blocks in $\cH$ can be partitioned into classes (or factors) $F_i$, such that every vertex of $G$ appears in exactly one block of each $F_i$. A resolvable $\cH$-decomposition is also referred to as an $\cH$-$factorization$ of $G$, whose classes are referred to as $\cH$-$factors$. Also, an $\cH$-decomposition is called $uniformly$ $resolvable$ if each class (or factor) $F_i$ consists of blocks that are all isomorphic to each other.

Recently, the existence problem for $\cH$-factorizations of $K_v$ has been studied, and results have been obtained. In the case of uniformly resolvable ${\cal H}$-decompositions, results have been given when $\cH$ is a set of two complete graphs of order  at most five in \cite{DLD, R, SG, WG};
when
$\cH$ is a set of two or three paths on two, three or four vertices in \cite{GM1,GM2, LMT}; for
$\cH =\{P_3, K_3+e\}$ in \cite{GM}; for $\cH =\{K_3, K_{1,3}\}$ in \cite{KMT}; for $\cH =\{C_4, P_{3}\}$ in \cite{M}; for $\cH =\{K_3, P_{3}\}$ in \cite{MT}; and for $\cH =\{K_2, K_{1,3}\}$ in \cite{CC}.

If $\cH=\{H_1,H_2 \}$, then there are many types of uniformly resolvable $\cH$-decompositions, depending on how many factors contain copies of $H_1$ and how many factors contain copies of $H_2$. We let $(H_1,H_2)$-$URD(v;r,s)$ denote a uniformly resolvable decomposition of $K_v$ into $r$ classes containing only copies of $H_1$ and $s$ classes containing only copies of $H_2$.

A $K_2$-factorization of $G$ is known as a 1-{\em factorization}, and its factors are called 1-{\em factors} with a $1$-factor denoted by $I$. It is well known that a 1-factorization of $K_v$  exists if and only if $v$ is even (\cite{Lu}).

We will consider the case of $H_1=K_2$ and $H_2=K_{1,n}$. While the general case $(K_2,K_{1,n})$-$URD(v;r,s)$ is still open and in progression, we have observed that the standard methods used for most cases of $(r,s)$ are not applicable to solve the cases when number of $1$-factors is small. Thus, we studied these cases separately. The case of $(K_2,K_{1,5})$-$URD(v;1,s)$ is presented in \cite{JAY}.

In general, if $r=1$, then $s=\frac{(v-2)(n+1)}{2n}$. In this paper, we focus on the $(K_2,K_{1,n})$-$URD(v;1,s)$. We completely solve the existence problem by proving the following result.\\

\noindent \textbf{Main Theorem.} 
{\em Let $n>1$ be an odd integer. A $(K_2,K_{1,n})$-$URD(v;1,s)$ exists if and only if $v \equiv 2(n+1) \pmod{n(n+1)}$}.

\section{Necessary Conditions}

\begin{lemma} 
\label{ness}
If a $(K_2,K_{1,n})$-$URD(v;1,s)$ exists with $n>1$ odd, then $v \equiv 2(n+1) \pmod{n(n+1)}$ and $s=\frac{((n+1)k+2)(n+1)}{2}$. 
\end{lemma}

\begin{proof}
Let $K_v$ be the complete graph on $v$ vertices, and $I$ be a $1$-factor of $K_v$. Because a $(K_2,K_{1,n})$-$URD(K_v;1,s)$ contains exactly one 1-factor, $I$, $v$ must be divisible by $2$. Also, if $s \geq 1$, $v$ must be divisible by $n+1$, and the total number of edges in $K_v-I$ must be divisible by the total number of edges in one $n$-star factor. Since $|E(K_v-I)|=\frac{v(v-1)}{2}-\frac{v}{2}$ and the total number of edges in one $n$-star factor is $\frac{nv}{(n+1)}$, we divide $\frac{v(v-2)}{2}$ by $\frac{nv}{(n+1)}$ to obtain $s=\frac{(v-2)(n+1)}{2n}$. Since $n$ cannot divide $n+1$, $(v-2)$ must be divisible by $n$. Therefore, we obtain the two congruences,

\begin{align}
    v \equiv& \textrm{ } 0 \pmod{n+1}\\
    v \equiv& \textrm{ } 2 \pmod{n}
\end{align}

By the chinese remainder theorem, we have $v \equiv 2(n+1) \pmod{n(n+1)}$.
\end{proof}

\section{Almost $n$-star Factors}
\label{ASF}
Let $S$ be a set of $v$ vertices, such that $v \equiv t \pmod{n+1}$. We will say a graph $G$ is {\em almost spanning} if it spans all but $t$ vertices in the set $S$. Define an {\em almost $n$-star factor} on a set of vertices $S$ to be an almost spanning graph on $S$ in which:
\begin{itemize}
    \item Each component of $S \backslash T$ is an $n$-star for some set, $T$ such that $|T|=t$, 
    \item the vertices in $T$ form a $(t-1)$-star.
\end{itemize}
We refer to the $(t-1)$-star as a {\em little star}. Note that if $t=1$, the little star is an isolated vertex, and if $t=0$, the almost $n$-star factor is equivalent to an $n$-star factor.

Let $G$ be a graph with $g$ vertices. The {\em difference} (or {\em length}) of the edge $e=\{u,v\}$ in $G$ with $u<v$, is $D(e)=min\{v-u,g-(v-u)\}$. If the difference of an edge $e$ is defined by $(v-u)$, then we will refer to this edge as a {\em forward edge}, and its difference will be called a {\em forward difference}. If the difference of an edge $e$ is defined by $g-(v-u)$, then we will refer to this edge as a {\em backward edge} (or {\em wrap-around edge}), and its difference will be called a {\em backward difference} (or {\em wrap-around difference}).

Let $F$ be an almost $n$-star factor.  Label the edges in $F$ by the differences they cover. Suppose each difference occurs no more than twice among the stars. If any difference $d$ appears exactly once, then use the label $d$($pure$). If any difference $d$ appears twice, then use the labels $d$($pure$) and $d'$($prime$) to distinguish them. Also, if $\{u,v\}$ is a forward edge with a prime difference, and $u<v$, then we will denote it by $\{u,v'\}$. If $\{u,v\}$ is a backward edge with a prime difference, and $u<v$, then we will denote it by $\{u,\overline{v'}\}$. We will refer to the corresponding differences as {\em pure differences} or {\em prime differences}, and similarly; we will refer to the corresponding edges as {\em pure edges} or {\em prime edges}. If a star consists of edges whose differences all have a pure label, we will refer to this star as a {\em pure star}. If it consists of edges whose differences all have a prime label, then it will be referred to as a {\em prime star}. If a star contains a mixture of pure edges and prime edges, then it will be referred to as a {\em mixed star}.

We aim to find a decomposition of $K_v-I$ into $n$-star factors, which is equivalent to showing the existence of a $(K_2,K_{1,n})-URD(v;1,s)$. Recall from our necessary condition that if a $(K_2,K_{1,n})-URD(v;1,s)$ exists, then $v \equiv 2(n+1) \pmod{n(n+1)}$. So let $v=n(n+1)k'+2(n+1)$ for some non negative integer $k'$. We will consider two cases, depending on the parity of $k'$. If $k'$ is odd, we let $k'=2k+1$ and write $v=2n(n+1)k+(n+1)(n+2)$. If $k'$ is even, we let $k'=2k$ and write $v=2n(n+1)k+2(n+1)$. Our approach for finding the desired decomposition will be to show that for any $k'$, we can construct almost $n$-star factors on a set of $g= \frac{v}{n+1}$ points in Section~\ref{ASF}. Then, in Section $4$, we will use these almost $n$-star factors to build $n$-star factors on a set of $v=g(n+1)$ points. In Section $5$, we show that the remaining edges of $K_v-I$ can be decomposed into $n$-star factors.

For any non negative integer, $t$, let $G$ be the complete graph $K_g$ with $V(G)$, $\{0,1,\dots, g-1\}$. We will find an almost $n$-star factor in $G$ with the little star having order $t$, where $0 \leq t \leq g-1$ and $t \equiv g \pmod{n+1}$. For the constructions that follow in this section, we define $\mu$ to be the maximum possible difference of any edge in $G$. Because $\mu$ is an integer, it should be noted that $\mu$ can also be used as a vertex label. Let $w \in \{1,2,\dots, n+1\}$ be the integer such that $\mu +1 \equiv w \pmod{n+1}$. If $\mu+1 = 0 \pmod{n+1}$, then we will choose $w$ to be $n+1$. Finally, because $n$ is odd, we write $n=2q+1$ for some positive integer $q$.

\subsection{$k'$ odd}
  In this section, we construct almost $n$-star factors when the number of isolated vertices is odd. In this case, $v= 2n(n+1)k + (n+1)(n+2)$, $g=2nk+(n+2)$, and $\mu = \frac{g-1}{2}$.

\begin{lemma}
\label{todd1}
Let $k=0$. There exists an almost $n$-star factor on $G$ with the following properties:
\begin{itemize}
\item Each forward difference $d \in \{1, 2, \dots, \frac{n+1}{2}\}$ appears at least once among the stars.
\item Each difference $d \in \{1, 2, \dots, \frac{n+1}{2}\}$ appears no more than twice among the stars.
\item There is one mixed star with $q+1$ pure edges and $q$ backward prime edges.
\item There is an isolated vertex.
\end{itemize}

\end{lemma}

\begin{proof}

Let $V(G) = \{0,1, \ldots, (n+1)\}$. Because $g=n+2$, we have that $t=1$. Therefore, there will be an isolated vertex. We give the mixed star $M$ and a single vertex $(2q+2)$:
\begin{align*}
  M&=(0; 1, 2, \dots, (q+1), \overline{(q+2)}',\overline{(q+3)}'\dots, \overline{(2q+1)}')
\end{align*}

Let $D=\{1, 2, \ldots, \frac{n+1}{2}\}$ denote the pure edge set. Because $\frac{n+1}{2} = \frac{2q+2}{2} = q+1$, every difference in $D$ is used at least once by the pure differences in $M$. Also, the $q$ backward prime differences used in $M$ are $\{\overline{2}',\overline{3}', \dots, \overline{q+1}'\}$, so every difference is used at most twice.

Thus, we have constructed an almost $n$-star factor with the desired properties.

\end{proof}

\begin{lemma}
\label{todd2}
Let $1 \leq k \leq q$. There exists an almost $n$-star factor on $G$ with the following properties:
\begin{itemize}
\item Each forward difference $d \in \{1, 2, \dots, \frac{2nk+(n+1)}{2}\}$ appears at least once among the stars.
\item Each difference $d \in \{1, 2, \dots, \frac{2nk+(n+1)}{2}\}$ appears no more than twice among the stars.
\item There is one mixed star with $q+1$ pure edges and $q$ prime edges where $(w-1)$ edges are backward prime edges and $q-(w-1)$ edges are forward prime edges.
\item There is one little star $L$ of size $t-1$.
\end{itemize}

\end{lemma}

\begin{proof}
Let $V(G) = \{0,1,\dots,2nk+n+1\}$ where $g=2nk+n+2$. In this case, we have $\mu=\frac{2nk+(n+1)}{2}$, $w= q+2-k$, and $|L|=t$. We describe the almost $n$-star factor on $G$ by giving a set of pure stars, $P_1$, a mixed star, $M$, a set of prime stars, $P_3$, and a little star, $L$.

\begin{align*}
  \textrm{Let } P_1&=\{(i-1); j, j-1, j-2, j-3, \dots, j-(n-1)\},  \\
&\mbox{where } i = 1, \ldots, k \mbox{ and } j=(\mu-w)-n\cdot (i-1).
\end{align*}

Let $D=\{1, 2, \ldots, \frac{2nk+n+1}{2}\}$  denote the set of all differences, and let $D_1$ be the differences used in $P_1$, that is $D_1 = \{1,2,3,\dots,(n+1)k-1\} \backslash \{(n+1)\delta | \delta = 1,2,\dots,k-1\}$. Because $P_1$ contains $k$ $n$-stars, $|D_1|$ is $nk$. Then, $D \backslash D_1$ will contain exactly $\frac{2nk+n+1}{2} - nk$ differences, which is equivalent to $\frac{n-1}{2}+1$. This is also equivalent to $w+(k-1)$ and $q+1$ by the following expression.
\begin{align*}
   \frac{2nk+n+1}{2} - nk =& \frac{2nk+n-1+2-2nk}{2} \\
        =& \frac{n-1}{2}+1 \\
        =& \frac{2q+1-1}{2}+1\\
        =& q+1\\
        =& (w-2+k)+1 \\ 
        =& w+(k-1).
\end{align*}
The $q+1$ unused pure differences will be used in the mixed star, $M$. This star will contain $q+1$ pure edges and $q$ prime edges.\\

Let $M$ be the mixed star with center $c=(\mu-w+1)$, which is the smallest available vertex after constructing the pure stars in $P_1$. We describe the leaves in $M$ by its set of edge lengths which we will denote by $l_M$. Let $l_M=A_1 \cup A_2 \cup B_1 \cup B_2$ where $A_1, A_2,B_1$, and $B_2$ are as follows. (Note that if $|A_1| = q+1$, $A_2$ will be empty, and in this case, we completely ignore the formula for $A_2$ below. Similarly, if $|B_1|=q$, $B_2$ will be empty, and similarly, we completely ignore the formula for $B_2$.)
\begin{align*}
    A_1&=\{(\mu-\delta)|\delta= 0,1,\dots,(w-1)  \}\\
    A_2&=\{(n+1)\delta|\delta= 1,2,\dots,(k-1)   \}\\
    B_1&=\{\overline{(\mu-\delta)'}|\delta= 0,1,\dots,(w-2)  \}\\
    B_2&=\{(\mu-w-\delta)'|\delta= 0,1,\dots,q-(w-1)-1  \}
    \end{align*}

Then, $M=(c;c+l_1,c+l_2,\dots,c+l_n)$ for each $l_i \in \{A_1 \cup A_2 \cup B_1 \cup B_2\}$. Note that $A_1$ and $A_2$ contain pure forward differences, $B_1$ contains prime backward differences, and $B_2$ contains prime forward differences. Furthermore, $|D \backslash D_1| = w + (k-1) = |A_1| + |A_2|$, and since $A_1 \cup A_2 \cup D_1 = D$, all differences in $D$ have appeared at least once. 

Construct $L$ by always choosing the next available smallest vertex for the center and the set of next available largest $t-1$ vertices for the leaves. Then, construct prime stars in $P_3$ by always choosing the next available smallest vertex for the center and the set of next available largest $n$ vertices for the leaves. This method ensures that all differences used in $L$ and $P_3$ are distinct. In fact, all vertex labels between $2\mu-w-q+1$ and $2nk+n+1$ were used in $M$, so the largest vertex label available for the longest leaf in $L$ was $2\mu-w-q$. The smallest vertex available for the center of $L$ was $\mu-w+2$. Thus, the largest possible difference in $L$ was $\mu-q-2$. The smallest possible difference that could be in $B_1 \cup B_2$ is $\mu-q$, and if $B_2$ is empty, the smallest possible difference that could be in $B_1$ is bigger than $\mu-q$. Thus, the largest possible difference in $L \cup P_3$ is smaller than the smallest possible prime difference in $M$. Therefore, every difference in $D$ appears at least once but at most twice.

An example of an almost $5$-star factor on $g=27$ vertices case is illustrated in Figure~\ref{v27}. Note that pure edges are colored red, prime edges in the mixed star are colored blue, prime edges in prime stars are colored light green, and the little star is colored black. \\

Thus, we have constructed an almost $n$-star factor with the desired properties.

\end{proof}

\begin{figure}[!htb]
    \centering
    \includegraphics[width=\linewidth]{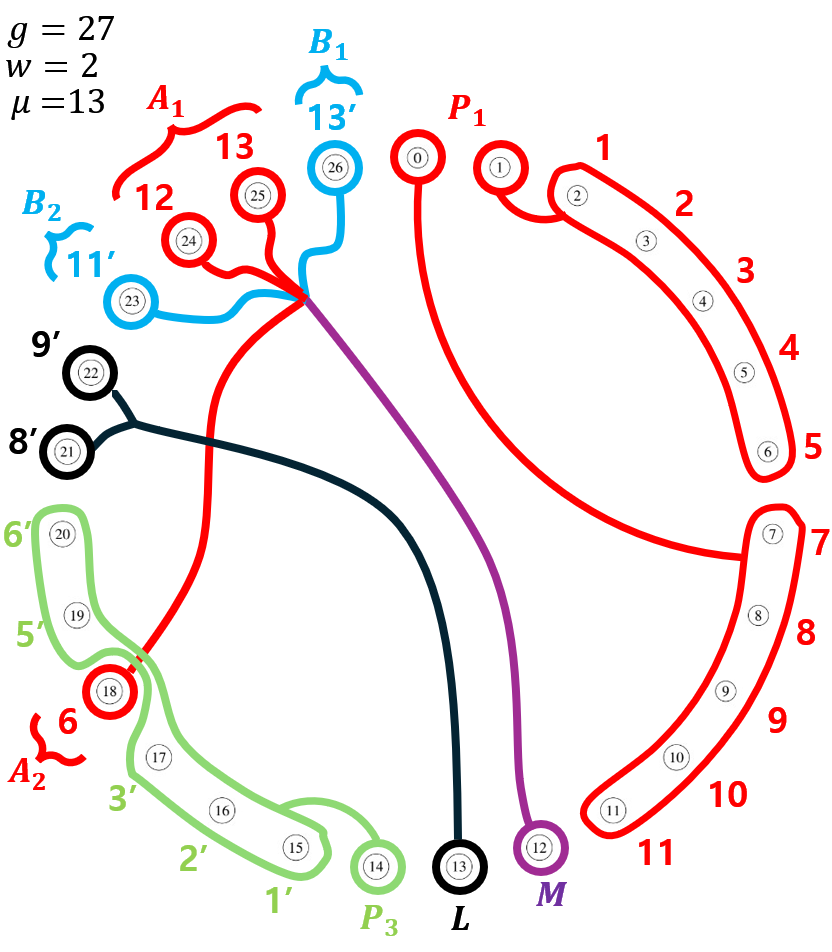}
    \caption{An almost $5$-star factor on $g=27$ vertices}
    \label{v27}
\end{figure}
\clearpage

\begin{lemma}
\label{todd3}
Let $k \geq q+1$. There exists an almost $n$-star factor on $G$ with the following properties:
\begin{itemize}
\item Each forward difference $d \in \{1, 2, \dots, \frac{2nk+(n+1)}{2}\}$ appears at least once among the stars.
\item Each forward difference $d \in \{1, 2, \dots, \frac{2nk+(n+1)}{2}\}$ appears no more than twice among the stars.
\item There is one mixed star with $q+1$ pure edges and $q$ foward prime edges.
\item There is one little star $L$ of size $t-1$.
\end{itemize}

\end{lemma}

\begin{proof}
Let $V(G)=\{0,1,\dots,2nk+n+1\}$. We have $\mu=\frac{2nk+(n+1)}{2}$, $|L|=t$ is an integer such that $0 \leq t \leq q-1$ and $t \equiv g\pmod{n+1}$, and $w$ is an integer such that $1 \leq w \leq n+1$ and $\mu+1 \equiv w \pmod{n+1}$. We describe the almost $n$-star factor on $G$ by giving a set of pure stars, $P_0, P_1, P_2$, a mixed star, $M$, a set of prime stars, $P_3$, and a little star, $L$.

\begin{align*}
  \textrm{Let } P_1&=\{(i-1); j, j-1, j-2, j-3, \dots, j-(n-1)\},  \\
&\mbox{where } i = 1, \ldots, \frac{\mu+1-w}{n+1} \mbox{ and } j=(\mu-w)-n\cdot (i-1).
\end{align*}

Let $D=\{1, 2, \ldots, \frac{2nk+n+1}{2}\}$ denote the set of differences. Let $D_1$ be the set of pure differences used in $P_1$ and $D_w = \{1,2,\dots, \mu-w\}$. Because $P_1$ contains exactly $\frac{\mu+1-w}{n+1}$ pure stars and the longest pure length contained in a pure star in $P_1$ is $(\mu-w)$, we have $|D \backslash D_w| = w$ and $|D_w \backslash D_1|= (\frac{\mu+1-w}{n+1}-1)$. Then, the set $D \backslash D_1$ will contain exactly  $w+(\frac{\mu+1-w}{n+1}-1)$ differences.

Next, we construct an additional pure star, $P_0$. It is clear to see that the vertex labels $\mu$ and $2\mu$ have not been used in any star in $P_1$. Therefore, we will define $P_0$ to have center at the vertex labeled $\mu$, and containing the edge $\{\mu,2\mu\}$. The remaining edges are defined to be the edges with center $\mu$ and leaves whose lengths constitute the next largest available length from $D \backslash D_1$. Thus, the pure edges contained in $P_0$ are all forward pure edges. Define $D_0$ to be the set of pure differences that were used for the pure star contained in $P_0$.

Now, we let $\rho_1 < \rho_2$ be the two smallest available (or unused) vertex labels after $P_1$ is constructed. We will choose $\rho_1$ to be the center of the mixed star $M$. The vertex labeled $\rho_1$ will depend on the relationship between $k$ and $q$.
\begin{align*}
    \rho_1 =
        \begin{cases}
            (\mu), &\textrm{ if } w=1 \textrm{ and } k = q+1,\\
           (\mu+1), &\textrm{ if } w=1 \textrm{ and } k > q+1,\\
           (\mu-w+1),  &\textrm{ if } w\not=1.
    \end{cases}
\end{align*}

Then, $M$ will be the mixed star, which contains $q+1$ pure edges and $q$ forward prime edges. The center of $M$ is $\rho_1$, and choose the $q+1$ largest available lengths in $D \backslash (D_0 \cup D_1)$ to obtain $q+1$ pure edges. Then, choose the $q$ largest possible vertex labels to obtain $q$ prime edges. Define $D_{m(pure)}$ be the set of differences which contains these $q+1$ pure differences and $D_{m(prime)}$ to be the set of prime differences that appear in $M$. We must show that $M$ does not conflict with the stars in $P_1 \cup P_0$.

First we consider the pure edges in $M$. If $P_0$ is not empty, then let the smallest difference in $P_0$ be $z$. The smallest vertex label on a leaf in $P_0$ is $(\mu+z)$. The largest pure difference in $M$ is $z-1$ if $w=n+1$ or $z-(n+1)$ if $w \not= n+1$. The vertex label for this longest leaf is
\begin{align*}
    \rho_1 + (z - 1) =& \textrm{ } (\mu-w+ z), \textrm{ if } w = n+1 \\
    &\qquad \qquad \textrm{               or }\\
    \rho_1 + (z - (n+1)) =&
        \begin{cases}
           (\mu+z-n), &\textrm{ if } w=1,\\
           (\mu+z-w-n),  &\textrm{ if } w \not = 1 \textrm{ and } w \not = n+1. 
    \end{cases}
\end{align*}
In each case, the vertex label with the longest pure edge in $M$ does not conflict with vertex labels in $P_0$. We must also show that the vertex label of the shortest pure edge in $M$ does not conflict with the center of $P_0$, which is labeled, $\mu$. If $w=1$, it is trivial because $\rho_1 \geq \mu$ or $P_0$ does not exist. If $w \not =1$, then because the shortest pure edge in $M$ has the smallest difference in $D \backslash (D_0 \cup D_1)$, (which is $(n+1)$), the vertex label for this leaf is $\rho_1 + (n+1)$. Then, because $w \not= 1$, we have
\begin{align*}
    \rho_1 + (n+1) =& \mu-w+1+(n+1) > \mu.
\end{align*}
Thus, the vertex label of the shortest pure edge in $M$ does not conflict with the center of $P_0$.

Now, consider the vertex labels on the $q$ prime edges in $M$. We will show that the differences covered are all distinct, and the edges are forward edges. Note that if $k=q+1$, then $P_0$ is empty and $\rho_1 = \mu$, so it is trivial. If $w=n+1$, then after constructing $P_1$ and $P_0$, the largest possible unused vertex is labeled $(2\mu-w+1)$. So, in this case, the largest possible prime difference in $M$ is 
\begin{align*}
    (2\mu-w+1) - \rho_1 = \mu.
\end{align*}
If $w \not =n+1$, the largest possible vertex labeling is $(2\mu-w)$. So, the largest possible prime difference in $M$ is
\begin{align*}
    (2\mu-w) - \rho_1 =
        \begin{cases}
           \mu-2, &\textrm{ if } w=1 \textrm{ and } k>q+1\\
           \mu-1,  &\textrm{ if } w \not =1 \textrm{ and } w \not = n+1.
    \end{cases}
\end{align*}
Hence, all the prime differences in $M$ are distinct forward prime differences.

Next, we will give the pure stars in $P_2$. First, we count the number of pure differences that have not yet appeared. Let $D_2 = D \backslash (D_0 \cup D_1 \cup D_{m(pure)})$. Because $|D|=\frac{2nk+n+1}{2}, n=2q+1$, $D_{m(pure)} = q+1$, $|D_1| = n\cdot \frac{\mu+1-w}{n+1}$, and $|D_0| = n$, we have:
\begin{align*}
    |D_2| =& |D| - |D_{m(pure)}| - |D_1| - |D_0| \\
                    =& \frac{2nk+n+1}{2} - (q+1) - n\cdot \frac{\mu+1-w}{n+1} - n\\
                    =& \frac{2k(2q+1)+(2q+1)+1)}{2} - (q+1) - n\cdot \frac{\mu+1-w}{n+1} - n\\
                    =& k(2q+1) - n\cdot \frac{\mu+1-w}{n+1} - n\\
                    =& k \cdot n - n\cdot \frac{\mu+1-w}{n+1} - n\\
                    =& n(k - (\frac{\mu+1-w}{n+1}+1)).
\end{align*}   

Since the number of remaining pure differences is divisible by $n$, this set of differences can be covered by our final set of pure stars, $P_2$. Note that if $k < 4+3q$, then $|D_2| = 0$, in which case $D_2$ will be empty.

We construct the stars in $P_2$ by always choosing the next available smallest vertex label for the center and the set of next available largest $n$ differences in $D \backslash (D_0 \cup D_1 \cup D_{m(pure)})$ for the differences of its branches.

We must show that the set of vertex labels used in $P_2$ does not conflict with any vertex labels used in $M$ or $P_0$. First, we will show that it does not conflict with the center of $P_0$ which is labeled, $\mu$. By the definition of $\rho_2$, it is guaranteed to be the smallest vertex label used in $P_2$. If $\rho_2 > \mu$ it is trivial. If $\rho_2 < (\mu)$, the smallest difference in $D_2$ is $(n+1)$, and we have $\mu - \rho_2 < w \leq n+1$ by the definition of $w$. Thus, $(\rho_2+(n+1))$ is the smallest possible vertex label of a leaf in $P_2$. Since $\rho_2+n+1 \geq \rho_2+w > \mu$, it is impossible for any star in $P_2$ to conflict with the center of $P_0$.

We must also show that the vertex labels used in $P_2$ do not conflict with any vertex label used in $M$. The smallest difference used in $M$ is a pure difference, and we chose it as the $q+1$th largest available difference from $D \backslash (D_0 \cup D_1)$. Let this smallest difference used in $M$ be $z_1$. Then, the largest difference in $D_2$, call it $z_2$, is clearly smaller than $z_1$. By the method used to construct the pure stars in $P_2$, $z_2$ must be the difference of a branch adjacent to $\rho_2$, and the vertex label of this leaf is $\rho_2+z_2$. This is the largest vertex label appearing in $P_2$. However, we also have $z_2 + (n+1) = z_1$ and
\begin{align*}
    \begin{cases}
           \rho_1 + 2  = \rho_2, &\textrm{ if } w=2,\\
           \rho_1 + 1  = \rho_2,  &\textrm{ otherwise. }
    \end{cases}
\end{align*}
Thus, we can conclude $\rho_2 + z_2 \not= \rho_1 + z_1$. Then, since the vertex labels used in $P_2$ do not conflict with any labels used in $M$ or $P_0$, it is possible to construct $P_2$ via the method we described above.\\

 To finish building our almost $n$-star factor, we will construct prime stars with distinct differences so that each difference in $D$ appears at most twice.

 To construct $L$, we will choose the next available smallest vertex label for the center and the set of next available largest $(t-1)$ vertices for the leaves. If $t=1$, we will choose the next available largest vertex label for the isolated vertex. If we build $L$ by the method described above, the longest possible prime difference is always smaller than the smallest prime difference in $D_{m(prime)}$ (which is the prime difference used in the mixed star). Finally, we construct a set of prime stars, $P_3$, by always choosing the next available smallest vertex for the center and the set of next available largest $n$ vertices for the leaves. Similarly, if we build the prime stars in $P_3$ by this method, all differences used in $L$ and the stars in $P_3$ are distinct. Hence, every difference in $D$ has appeared at least once but at most twice.\\

We have constructed an almost $n$-star factor with the desired properties.

\end{proof}

\subsection{$k'$ even}
In this section, we construct almost $n$-star factors when the order of $L$ is even. In this case, $v= 2n(n+1)k + 2(n+1)$, $g=2nk+2$, and $\mu = \frac{g}{2}$.

\begin{lemma}
\label{teven1}
Let $k=1$. There exists an almost $n$-star factor on $G$ with the following properties:
\begin{itemize}
\item Each forward difference $d \in \{1, 2, \dots, n\}$ appears at least once among the stars.
\item Each forward difference $d \in \{1, 2, \dots, n\}$ appears no more than twice among the stars.
\item There is no mixed star.
\item There is no little star.
\end{itemize}

\end{lemma}

\begin{proof}

$V(G)=\{0,1, \ldots, (2n+1)\}$. We construct a pure star $P_1$ and a prime star $P_2$ as follows.
\begin{align*}
  P_1&=(0; 1, 2, \dots, n)\\
  P_2&=((n+1); (n+2)', (n+3)', \dots, (2n+1)')
\end{align*}

Let $D=\{1, 2, \ldots, n \}$ be the set of differences. It is clear that the set of differences covered by the edges contained in the pure star $P_1$ are exactly equal to $D$. Similarly, the set of differences covered by the edges contained in the prime star $P_2$ are exactly equal to $D$.

Thus, we have constructed an almost $n$-star factor with the desired properties.

\end{proof}

\begin{lemma}
\label{teven2}
Let $k \geq 2$. There exists an almost $n$-star factor on $G$ with the following properties:
\begin{itemize}
\item Each forward difference $d \in \{1, 2, \dots, nk\}$ appears at least once among the stars.
\item Each forward difference $d \in \{1, 2, \dots, nk\}$ appears no more than twice among the stars.
\item There is one little star $L$ of size $t-1$. If $t=0$, there is no little star.
\end{itemize}

\end{lemma}

\begin{proof}
Let Let $V(G) = \{0,1,\dots,2nk+1\}$ where $g=2nk+2$. We have $\mu=nk+1$, $t$ is an integer such that $0 \leq t \leq g-1$ and $t \equiv g \pmod{n+1}$, and $w$ is an integer such that $1 \leq w \leq n+1$ and $\mu+1 \equiv w \pmod{n+1}$. We construct the following sets of pure stars, $P_0$, $P_1$, $P_2$, a set of prime stars $P_3$, and a little star $L$, on $G$.

\begin{align*}
  \textrm{Let } P_1&=\{(i-1); j, j-1, j-2, j-3, \dots, j-(n-1)\},  \\
&\mbox{where } i = 1, \ldots, \frac{\mu+1-w}{n+1} \mbox{ and } j=(\mu-w)-n\cdot (i-1).
\end{align*}

Let $D=\{1, 2, \ldots, nk\}$  denote the set of all differences. Let $D_1$ be the pure lengths used in $P_1$. Since $P_1$ contains exactly $\frac{\mu+1-w}{n+1}$ pure stars and the longest pure length contained in a pure star in $P_1$ is $(\mu-w)$, the set $D \backslash D_1$ will contain exactly  $w+(\frac{\mu+1-w}{n+1}-1)$ pure differences.

Next, we construct $P_0$ which will contain exactly one pure star. We define the pure star in $P_0$ to have center at the vertex labeled $\mu$ and leaves defined by forward edges with the $n$ largest available differences from $D \backslash D_1$. The largest vertex label used in $P_1$ was $(\mu-w)$, and thus, the center of $P_0$, $\mu$, had not been used. The largest difference in $D \backslash D_0$ is $nk$, so the endvertex of the edge with this difference will be $\mu + nk = nk+1 + nk = 2nk+1$ which is an available vertex label on $V$. Hence, it is possible to construct $P_0$ via the method described above. Define $D_0$ to be the set of pure differences that were used for the single pure star in $P_0$.

We must construct one more set of pure stars $P_2$ so that every difference in $D$ appears at least once. Note that if $k < n+3$, $P_2$ will be empty because $D = D_1 \cup D_0$. If $k \geq n+3$, we will construct as many pure stars as needed by always choosing the next available smallest vertex labeling that is greater than $(\mu)$ for the center. The set of next available largest $n$ differences in $D \backslash (D_0 \cup D_1)$ will be used to define each star's set of leaves. Let $D_2$ be the set of pure differences that covered by the stars in $P_2$. Then, we have $D = (D_0 \cup D_1 \cup D_{2})$.

To prove that we can construct $P_2$ via the method described above, we must first show that $D \backslash (D_0 \cup D_1)$ is divisible by $n$. However, since we have $|D| = 2nk$ and $|D_0 \cup D_1|$ is clearly divisible by $n$, it is obvious that $|D \backslash (D_0 \cup D_1)|$ is divisible by $n$. Finally, we will prove that the set of vertex labels used in $P_2$ does not conflict with any vertex labels appearing in $P_0$ or $P_1$. The smallest vertex label used in $P_2$ is $(\mu+1)$, thus it is impossible to conflict with any vertex label used in $P_1$. It is also obvious that it does not conflict with the center of $P_0$, which is $\mu$. At last, we will show that $P_2$ does not conflict with any vertex label of leaves in $P_0$. The smallest difference used in $P_0$ is a pure difference, and we chose it as the $n$-th largest available difference from $D \backslash (D_1)$. Let this smallest difference used in $P_0$ be $z_1$. Then, the largest difference in $D_2$, call it $z_2$, is clearly smaller than $z_1$. By the method used to construct the pure stars in $P_2$, $z_2$ must be the difference of a leaf adjacent to $(\mu+1)$, and the vertex label of this leaf is $(\mu+1)+z_2$. This is the largest vertex label appearing in $P_2$. However, because we have $z_2 + (n+1) = z_1$, it is clear that the largest vertex label in $P_2$, $(\mu+1)+z_2$, is impossible to conflict with the smallest vertex label in $P_0$, $\mu+z_1$. Thus, we can conclude that it is possible to construct a set of pure stars $P_2$ via the method we described above.\\

By constructing $P_0, P_1, P_2,$ with pure differences $(D_0 \cup D_1 \cup D_{2}) = D$, we have shown that each difference in $D$ has appeared exactly once. Now, to finish building the almost $n$-star factor, we will construct prime stars which contain distinct forward differences so that each difference in $D$ appears at most twice. To complete our construction, we will construct the little star $L$ first, then construct the set of prime stars $P_3$.
 
 To construct $L$, we will choose the next available smallest vertex label for the center and the set of next available largest $(t-1)$ vertices for the leaves. If we build $L$ in this way, the vertex labeling of the center will be $(\mu-w+1)$, and the vertex label of its longest leaf will be $(2\mu-w)$. Thus, the largest difference contained in $L$ will be $(2\mu-w)-(\mu-w+1) = \mu-2 = nk-1$. Then, since $nk-1$ is in $D$, we've shown that all differences in $L$ are forward prime differences.
 
 We construct prime stars in $P_3$ by always choosing the next available smallest vertex for the center and the set of next available largest $n$ vertices for the leaves. If we build the prime stars in $P_3$ by this method, all differences used in $P_3$ will be smaller than any of the differences used in $L$. Hence, every difference in $D$ will have appeared at least once but at most twice.

We have constructed an almost $n$-star factor with the desired properties. \\

\end{proof}


\section{Part I factors}
In this section, we use the almost $n$-star factors of $G$ that were built in Section $3$ to build $n$-star factors of $K_v$. Recall that $v=n(n+1)k' + 2(n+1)$ for some non-negative integer $k'$. In Lemmas ~\ref{todd1} -~\ref{teven2}, we showed that for any $k'$, we can construct an almost $n$-star factor on $G=K_g$ which $g=\frac{v}{n+1}$. The following result will show how to obtain $g$ $n$-star factors on $K_v$.

\begin{lemma} 
\label{Part I}
(Part I factors) If there exists an almost $n$-star factor on $G$ with the properties described in Lemmas~\ref{todd1} -~\ref{teven2}, then there exists $g$ $n$-star factors of $K_v$.
\end{lemma}

\begin{proof}

Let $V = \{0,1,2,\dots, v-1 \}$ be the vertex set of $K_v$, and let $F$ be the almost $n$-star factor constructed from Lemmas~\ref{todd1} -~\ref{teven2} on  $g=\frac{v}{n+1}$ vertices. We partition $V$ into $n+1$ subsets as follows. Let $V = \bigcup_{i=0}^n V_i$ where $V_i=\{\nu \in V | \nu\equiv i \pmod{(n+1)} \}$. \\

{\bf Case 1: $k'$ odd} 

Let $k'=2k+1$, so $g=2nk+(n+2)$, with $k \geq 0$. For any $k \geq 0$, consider the almost $n$-star factor constructed in Lemmas~\ref{todd1} -~\ref{todd3}. Recall that $F$ consists of pure stars, prime stars, a mixed star, and a little star.

For each star, let $c$ be the vertex label of the center vertex, and $l_i$ or $l_i'$ with $i=1,2,\dots,n$ to be the vertex labels of the leaves. As described in Section $3$, a leaf $l_i$ is an end vertex of an edge with a pure difference, and a leaf $l_j'$ is an end vertex of an edge with a prime difference. Note that $l_i < l_j$ and $l_i' < l_j'$ if $i < j$. However, it need not be that $l_i < l_j'$ if $i < j$.\\

For each pure star $s=(c; l_1, l_2, \dots, l_n) \in F$, construct $n+1$ stars 
$s_i=((n+1)c+i; 
(n+1)l_1+i, 
(n+1)l_2+i, 
\dots, 
(n+1)l_n+i)$ for $i=0,1,2,3,\dots, n$. By constructing these stars, all vertices in the set $\{c(n+1)+i, \l_1(n+1)+i, l_2(n+1)+i, \dots, l_n(n+1)+i | i \in \mathbb{Z}_n \}$, $\forall s \in F$ have been used. An example for $i=0,1$ is illustrated as red stars in Figure~\ref{v162}. \\

For each prime star $p=(c; l_1', l_2', \dots, l_n') \in F$, construct $n+1$ stars 
$p_i=((n+1)c+i; 
(n+1)l_1'+((n+i) \pmod{(n+1)}), 
(n+1)l_2'+((n-1+i) \pmod{(n+1)}), 
(n+1)l_3'+((n-2+i) \pmod{(n+1)}), 
\dots, 
(n+1)l_5'+((1+i) \pmod{(n+1)}))$ for $i=0,1,2,\dots,n$. By constructing these stars, all vertices in the vertex set $\{c(n+1)+i, \l_1(n+1)+i, l_2(n+1)+i, \dots, l_n(n+1)+i |  i \in \{0,1,\dots,n\} \}$, $\forall p \in F$ have been used on one of $p_i$. An example for $i=0,1$ is illustrated as light green stars in Figure~\ref{v162}.\\

Before we construct the mixed star, we introduce a symbol $\oplus$. If $\oplus$ is in an expression containing a label $l_i'$ representing a forward prime edge, then $\oplus$ will be replaced by $+$, and if $\oplus$ is in an expression containing a label $l_i'$ representing a backward prime edge, then $\oplus$ will be replaced by $-$. Now, the mixed star has $q+1$ pure edges and $q$ prime edges. For the mixed star $m=(c; l_1, l_2, \dots, l_{q+1}, l_1', l_2',\dots,l_q') \in F$, construct $n+1$ stars 
$m_{i}=((n+1)c+i; 
(n+1)l_{1}+i, 
(n+1)l_{2}+i,
\dots,
(n+1)l_{q+1}+i, 
(n+1)l_{1}'+((i\oplus1) \pmod{(n+1)}), 
(n+1)l_{2}'+((i\oplus2) \pmod{(n+1)}), 
\dots,
(n+1)l_{q}'+((i\oplus q) \pmod{(n+1)}))$ for $i=0,1,2,\dots,n$. By constructing these stars, all vertices in the vertex set $\{c(n+1)+i, \l_1(n+1)+i, l_2(n+1)+i, \dots, l_n(n+1)+i | \forall i \in \{0,1,\dots,n\}, M \in F  \}$ have been used on one of $m_i$. An example for $i=0,1$ is illustrated as red (pure) and blue (prime) endvertices with purple edges in Figure~\ref{v162}.\\\\

For the little star $L = (c;l_1,l_2,\dots,l_{t-1}) \in F$, we will construct $t$ $n$-stars $L_i$ for $i=0,1,\dots,t-1$ based on following conditions.\\

If $t = 1$ with $L = \{ c \}$, we construct $L_0 = ( (n+1)c ; (n+1)c+1, (n+1)c+2, \dots, (n+1)c+n) $. \\

Note that if $n=3$, the only possibilities for odd $t$ are $1$ or $3$, so we will construct these cases separately as below. 

If $n=3$ and $t=1$, construct $L_0 = ( 4c ; 4c+1, 4c+2, 4c+3) $. 

If $n=3$ and $t=3$, construct $L_0, L_1,$ and $L_2$ as follows: 
\begin{align*}
            L_0 =& (4c ; 4l_1+1, 4l_1+2, 4l_1+3)\\
            L_1 =& (4c+1 ; 4c+3, 4l_1 + 0, 4l_2 + 2)\\
            L_2 =& (4c+2 ; 4l_2+0, 4l_2+1, 4l_2+3)
\end{align*}

If $1 < t < q+2$, we first construct $t-2$ stars $L_i$ for $i=0, \dots, t-3$.
\begin{align*}
    L_i = ((n+1)c+i \textrm{ ; } &(n+1)l_{i+1}+(i+1)\pmod{(n+1)}, \\
&(n+1)l_{i+1}+(i+2)\pmod{(n+1)}, \\
&(n+1)l_{i+1}+(i+3)\pmod{(n+1)}, \\
&\qquad \qquad \vdots\\
&(n+1)l_{i+1}+(i+n)\pmod{(n+1)}) 
\end{align*}

Then, we have two more stars $L_{t-2}$ and $L_{t-1}$:
\begin{align*}
    L_{t-2} = ((n+1)c &+ (t-2) \textrm{ ; }\\
    &(n+1)l_1 + 0 ,(n+1)l_2 + 1 , \dots, (n+1)l_{t-2} + t-3,\\
        &(n+1)c + (t-1),(n+1)c + (t-0),\dots,(n+1)c + (q),\\
        &(n+1)l_{t-1} + (q+1),\\
        &(n+1)c + (q+2),(n+1)c + (q+3),\dots,(n+1)c + (n) )
\end{align*}

 and 

\begin{align*}
    L_{t-1}& = ((n+1)c + (q+1) \textrm{ ; } \\
&(n+1)l_{t-1} + (0),
(n+1)l_{t-1} + (1),
(n+1)l_{t-1} + (2),
\dots,
(n+1)l_{t-1} + (q),\\
&(n+1)l_{t-1} + (q+2),
(n+1)l_{t-1} + (q+3),
\dots,
(n+1)l_{t-1} + (n) ).
\end{align*}
\\

If $t = q+2$, we first construct $t-2$ $n$-stars $L_i$ for $i=0, \dots, t-3$.
\begin{align*}
    L_i = ((n+1)c+i \textrm{ ; } &(n+1)l_{i+1}+(i+1)\pmod{(n+1)}, \\
&(n+1)l_{i+1}+(i+2)\pmod{(n+1)}, \\
&(n+1)l_{i+1}+(i+3)\pmod{(n+1)}, \\
&\qquad \qquad \vdots\\
&(n+1)l_{i+1}+(i+n)\pmod{(n+1)}) 
\end{align*}

Then, we have two stars $L_{t-2}$ and $L_{t-1}$:
\begin{align*}
    L_{t-2} &= ((n+1)c+(t-2) \textrm{ ; } \\ 
&(n+1)l_1 + 0,
(n+1)l_2 + 1,
(n+1)l_3 + 2,
\dots,
(n+1)l_{t-2} + (t-3),\\
&(n+1)l_{t-1} + (t-1),\\
&(n+1)c+(t),
(n+1)c+(t+1),
(n+1)c+(t+2),
\dots,
(n+1)c+(n) )
\end{align*}

and 
\begin{align*}
    L_{t-1} &= ((n+1)c+(q+1) \textrm{ ; } \\
&(n+1)l_{t-1} + 0,
(n+1)l_{t-1} + 1,
(n+1)l_{t-1} + 2,
\dots,
(n+1)l_{t-1} + (q+0),\\
&(n+1)l_{t-1} + (q+2),
(n+1)l_{t-1} + (q+3),
\dots,
(n+1)l_{t-1} + (n) ).
\end{align*}
\\

If $t > q+2$, then we construct $t-1$ $n$-stars $L_i$ for $i=0, \dots, t-2$.  
\begin{align*}
    L_i = ( (n+1)c+i \textrm{ ; } &(n+1)l_{i+1}+(i+1)\pmod{(n+1)}, \\
&(n+1)l_{i+1}+(i+2)\pmod{(n+1)}, \\
&(n+1)l_{i+1}+(i+3)\pmod{(n+1)}, \\
& \qquad \qquad \vdots\\
&(n+1)l_{i+1}+(i+n)\pmod{(n+1)} )
\end{align*}

Then, we construct a star $L_{t-1}$,
\begin{align*}
    L_{t-1} &= ((n+1)c+(t-1) \textrm{ ; } \\
&(n+1)l_1 + 0,
(n+1)l_2 + 1,
(n+1)l_3 + 2,
\dots,
(n+1)l_{t-1} + (t-2),\\
&(n+1)c+(t+0),
(n+1)c+(t+1),
(n+1)c+(t+2),
\dots,
(n+1)c+(n) )
\end{align*}

By constructing these stars, all vertices in the vertex set $\{c(n+1)+i, \l_1(n+1)+i, l_2(n+1)+i, \dots, l_{t-1}(n+1)+i | \forall i \in \{0,1,\dots,t-1\}, \forall L \in F   \}$ have been used on one of $L_i$. 
An example is illustrated as black stars in Figure~\ref{v162}.\\

Thus, by constructing $n$-stars via the method above, we successfully constructed an $n$-star factor on $V = \{0,1,\dots,g(n+1)-1\}$, $B$. We obtain a total of $g$ $n$-star factors by taking $B+(n+1)j$ for $j=0,1,\dots, g-1$.
\\

\begin{figure}[!htb]
    \centering
    \includegraphics[width=\linewidth]{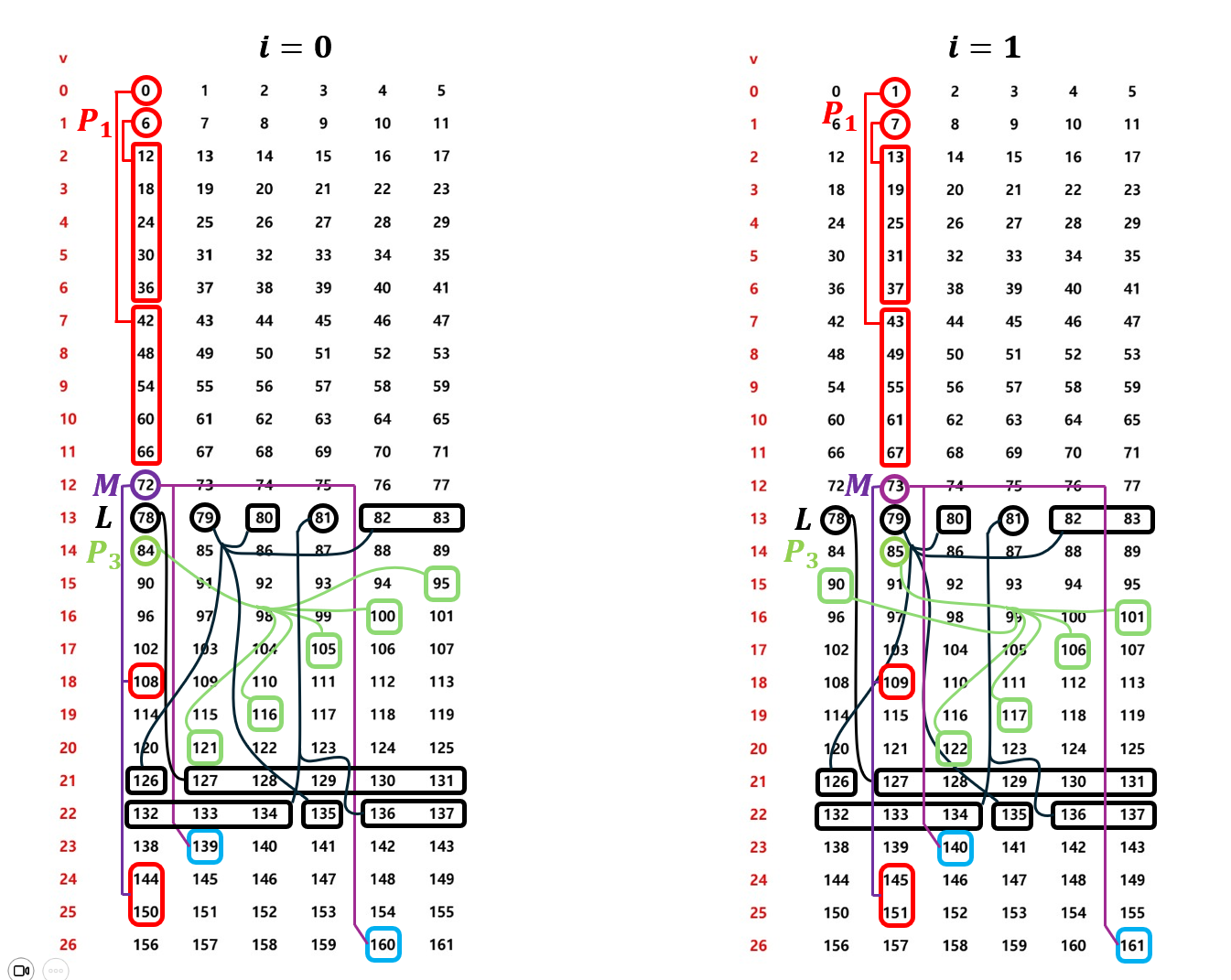}
    \caption{Illustration of a partial $5$-star factor on $v=162$ vertices from the almost $5$-star factor on $27$ vertices in Figure$~\ref{v27}$}
    \label{v162}
\end{figure}



{\bf Case $2:$ $k'$ even:} 

Let $k'=2k$, so $g=2nk+2$, with $k \geq 0$. \\

 For any $k \geq 0$, let $F$ be the almost $n$-star factor constructed from Lemma~\ref{teven1} or Lemma~\ref{teven2}. Recall that $F$ consists of pure stars, prime stars, and a little star. 

For each star, let $c$ be the vertex label of the center vertex, and $l_i$ or $l_i'$ with $i=1,2,\dots,n$ to be the vertex labels of the leaves. As described in Section $3$, a leaf $l_i$ is an end vertex of an edge with a pure difference, and a leaf $l_j'$ is an end vertex of an edge with a prime difference. Note that $l_i < l_j$ and $l_i' < l_j'$ if $i < j$.\\

For each pure star $s=(c; l_1, l_2, \dots, l_n) \in F$, construct $n+1$ stars 
$s_i=((n+1)c+i; 
(n+1)l_1+i, 
(n+1)l_2+i, 
\dots, 
(n+1)l_n+i)$ for $i=0,1,2,3,\dots, n$. By constructing these stars, all vertices in the vertex set $\{c(n+1)+i, \l_1(n+1)+i, l_2(n+1)+i, \dots, l_n(n+1)+i | i \in \{0,1,\dots,n\} \}$, $\forall s \in F $ have been used on one $s_i$.\\

For each prime star $p=(c; l_1', l_2', \dots, l_n') \in F$, construct $n+1$ stars 
$p_i=((n+1)c+i; 
(n+1)l_1'+((n+i) \pmod{(n+1)}), 
(n+1)l_2'+((n-1+i) \pmod{(n+1)}), 
(n+1)l_3'+((n-2+i) \pmod{(n+1)}), 
\dots, 
(n+1)l_5'+((1+i) \pmod{(n+1)}))$ for $i=0,1,2,\dots,n$. By constructing these stars, all vertices in the vertex set $\{c(n+1)+i, \l_1(n+1)+i, l_2(n+1)+i, \dots, l_n(n+1)+i |  i \in \{0,1,\dots,n\} \}$, $\forall p \in F$ have been used on one of $p_i$.\\

Then, for the little star $L = (c;l_1,l_2,\dots,l_{t-1}) \in F$, we first construct $t-1$ $n$-stars $L_i$ for $i=0, \dots, t-2$ as follows:
\begin{align*}
    L_i = ( (n+1)c+i \textrm{ ; } &(n+1)l_{i+1}+(i+1)\pmod{(n+1)}, \\
&(n+1)l_{i+1}+(i+2)\pmod{(n+1)}, \\
&(n+1)l_{i+1}+(i+3)\pmod{(n+1)}, \\
& \qquad \qquad \vdots\\
&(n+1)l_{i+1}+(i+n)\pmod{(n+1)} )
\end{align*}

Then, we construct a star $L_{t-1}$ as follow:
\begin{align*}
    L_{t-1} &= ((n+1)c+(t-1) \textrm{ ; } \\
&(n+1)l_1 + 0,
(n+1)l_2 + 1,
(n+1)l_3 + 2,
\dots,
(n+1)l_{t-1} + (t-2),\\
&(n+1)c+(t+0),
(n+1)c+(t+1),
(n+1)c+(t+2),
\dots,
(n+1)c+(n) )
\end{align*}

By constructing these stars, all vertices in the vertex set $\{c(n+1)+i, \l_1(n+1)+i, l_2(n+1)+i, \dots, l_{t-1}(n+1)+i |  i \in \{0,1,\dots,t-1\} \},$ $\forall L \in F$ have been used on one $L_i$.\\

By constructing $n$-stars via the method above, we successfully constructed $n$-star factor on $V = \{0,1,\dots,g(n+1)-1\}$, $B$. We obtain a total of $g$ $n$-star factors by taking $B+(n+1)j$ for $j = 0,1,\dots,g-1$.\\


\end{proof}


\section{Balanced Star Arrays and Part II factors}

In the graph $K_v$, for each difference $d$, there are $v$ edges with that difference. So when decomposing $K_v-I$ into $n$-star factors, we must ensure that for any difference $d$, each edge with difference $d$ appears exactly once in a star. To keep track of which differences were used in the Part I factors, and which differences we still need to cover to complete the decomposition, we will use an array with special properties.\\

Let $V$ be a set of $v$ vertices, $D= \{0,1,\dots,\frac{v-2}{2}$, and $D' = D\backslash \{d \in D| d\equiv 0 \pmod{n+1} \}$. A {\em balanced star array} for $V$ is a $\lceil \frac{v-2}{2(n+1)}\rceil \times n$ array $T=T^{1} \cup T^{2}$ whose entries partition the set $D'$, and satisfy the following properties:

\begin{itemize}
\item The columns are indexed by $j= 1, 2, \ldots, n$, and all entries in the $j^{th}$ column are congruent to $j \pmod{(n+1)}$.
\item $T^{1}$ is a subarray of $T$ whose entries represent the differences covered by the stars in the Part I factors, and one row of $T^1$ contains $q$ filled and $n-q$ empty cells where $r$ is the remainder when $\frac{v-2}{2}$ is divided by $(n+1)$.
\item $T^{2}$ is a subarray of $T$ with no empty cells.
\end{itemize}

    Each entry $d$ in $T$ represents all edges $\{u,v\}$, with difference $d$ such that $u<v$ and $u \in V$. The entries in $T^1$ are differences from $D'$ that have been covered in the Part I factors, and the entries in $T^2$ are differences that have not yet been covered.
    
    Note that none of the differences in $D'$ are congruent to $0 \pmod{n+1}$. This is because, by Lemma~\ref{Part I}, developing the base block $B$ guarantees that every vertex $u$ in $V$ is incident to all edges with difference $d \equiv 0 \pmod{n+1}$. Thus, we are only concerned with the differences that are covered by prime edges. We will build the arrays so that each full row of $T^1$ corresponds to the set of $n$ differences covered by a particular Part I prime star. If $\frac{v-2}{2}$ is not divisible by $n+1$, then one row of $T^1$ will contain exactly $r$ cells where $r$ is the remainder when $\frac{v-2}{2}$ is divided  by $n+1$. Note that the non-empty cells in this row corresponds to the prime edges in the mixed star from Part I. \\

\begin{lemma} 
\label{Part II}
(Part II factors) If there exists a balanced star array for each set $V_{i}, i \in \mathbb{Z}_{n+1}$, then there is a decomposition of $K_{v}-I$ into $n$-star factors.
\end{lemma}

\begin{proof}
Let the $1$-factor $I$ be given by:
\begin{align*}
    I = \{ \{ i,i+\frac{v}{2} \} : i \in \{0,1,2,\dots, \frac{v}{2}-1 \} \}.
\end{align*}

Every edge with difference $d \equiv 0 \pmod{n+1}$ is contained in exactly one pure or mixed star from Part I. Therefore, we need only be concerned with ensuring that each edge with difference $d \not \equiv 0 \pmod{n+1}$ is contained in exactly one $n$-star. Let $V = \bigcup_{i=0}^n V_i$, and let $T_{i}$ be the balanced star array for the set $V_{i}$. The differences in $T_{i}^{1}$ are covered by the factors given in Part I. 

For each row of the subarray $T_{i}^{2}$, we construct an $n$-star factor as follows. Let the entries in the given row be $(d_1, d_2, d_3, \dots, d_n)$. Define the base star to be $s=(i;i+d_1, i+d_2, i+d_3, \dots, i+d_n)$. We obtain $\frac{v-(n+1)}{n+1}$ more stars by taking $s+(n+1)j$ for $j=1,2, \ldots, \frac{v-(n+1)}{n+1}$. Because each $d_k \equiv k \pmod{n+1}$, for $k=1,2,3,\dots,n$,  we are guaranteed that these stars are disjoint and will span the set $V$. Furthermore, each forward edge of difference $d_k$ on the vertices of $V_{i}$ has been covered exactly once by this $n$-star factor. Because the balanced star array for $V_{i}$ partitions $D'$, we have exhausted all of the edges $\{u,v\}$ with difference $d$ such that $u<v$ and $u \in V_i$. Because there is a balanced star array for each $V_{i}$, we have covered all edges of each difference. Thus we have decomposed $K_{v}-I$ into $n$-star factors.
\end{proof}

Next, we build the needed balanced star arrays, used to record the differences covered by the stars from the Part I factors given by Lemma~\ref{Part I}. Lemma~\ref{BSAo} deals with the case when $k'$ is odd, and Lemma~\ref{BSAe} is for when $k'$ is even.\\


\begin{lemma} 
\label{BSAo}
There is a balanced star array for each $V_{i}$, $i \in \mathbb{Z}_{n+1}$ when $v=2n(n+1)k+(n+1)(n+2)$ with $k \geq 0$.
\end{lemma}

\begin{proof}
If $v=2n(n+1)k+(n+1)(n+2)$, observe that $T$ has $\lceil \frac{v}{2(n+1)} \rceil$ rows, and one row contains $q$ entries because $\frac{v-2}{2} \equiv q \pmod{n+1}$. It also means that the first $q$ columns contain one more entry than other columns. For $T=T^1 \cup T^2$ to be a balanced star array, we must prove that the subarray $T^1$ contains a row with $q$ entries and the subarray $T^2$ contains no empty cells.\\

For each $i \in \{0,1,\dots,n\}$, let $T_i=T_i^1 \cup T_i^2$ be the array for $V_i$, recording the differences used in the Part I factors from Lemma~\ref{Part I}, when $k'$ is odd. We will prove that each $T_i$ is balanced. \\

Note that all itemized expressions below are$\pmod{n+1}$.\\

For any prime star $p_i$, the $n$ forward differences are as follows:
\begin{itemize}
    \item $(((n+1)l'_1 + n + i ) - ((n+1)c+i) ) \pmod{n+1)} \equiv n $
    \item $(((n+1)l'_2 + (n-1) + i ) - ((n+1)c+i) ) \pmod{n+1)} \equiv n-1$
    \item $(((n+1)l'_3 + (n-2) + i ) - ((n+1)c+i) ) \pmod{n+1)} \equiv n-2$\\
     $\vdots$ 
    \item $(((n+1)l'_{n-1} + 2 + i ) - ((n+1)c+i) ) \pmod{n+1)} \equiv 2$
    \item $(((n+1)l'_{n} + 1 + i ) - ((n+1)c+i) ) \pmod{n+1)} \equiv 1$
\end{itemize}

We can clearly see that each prime star $p_i$ covers $n$ differences, which are equivalent to $1,2,\dots,n \pmod{(n+1)}$. Thus, each prime star's differences can be recorded as a row of $T_i^1$ with no empty cells.\\

When recording the differences covered by little stars, we consider cases, based on $t$, which gives the number of vertices in $L$.

Case $t = 1$: When $L = \{ c \}$, we have the following $n$ differences.
\begin{itemize}
    \item $( ((n+1)c+1) - ((n+1)c ) )   \pmod{(n+1)} \equiv 1$
    \item $( ((n+1)c+2) - ((n+1)c ) )   \pmod{(n+1)} \equiv 2$
    \item $( ((n+1)c+3) - ((n+1)c ) )   \pmod{(n+1)} \equiv 3$\\
     $\vdots$
    \item $( ((n+1)c+n) - ((n+1)c ) )   \pmod{(n+1)} \equiv n$    
\end{itemize}

Thus, the differences covered by this star can be recorded as a row of $T_i^1$ with no empty cells.

Case $1 < t < q+2$: For $i=0,1,\dots,t-3$, the $n$ forward prime differences in $L_i$ are:

\begin{itemize}
    \item $( ( (n+1)l_{i+1}+(i+1)) - ( (n+1)c+i ) )  \pmod{(n+1)} \equiv 1$
    \item $( ( (n+1)l_{i+1}+(i+2)) - ( (n+1)c+i ) )  \pmod{(n+1)} \equiv 2$
    \item $( ( (n+1)l_{i+1}+(i+3)) - ( (n+1)c+i ) )  \pmod{(n+1)} \equiv 3$\\
     $\vdots$
    \item $( ( (n+1)l_{i+1}+(i+n)) - ( (n+1)c+i ) )  \pmod{(n+1)} \equiv n$
\end{itemize}

So the differences covered by $L_i$ can be recorded as a row of $T_i^1$ with no empty cells.

The differences covered by $L_{t-2}$ are:

\begin{itemize}
    \item $( ( (n+1)l_1 + 0 ) - ((n+1)c + (t-2) ) )  \pmod{(n+1)} \equiv 2-t$
    \item $( ( (n+1)l_2 + 1 ) - ((n+1)c + (t-2) ) )  \pmod{(n+1)} \equiv 3-t$\\
     $\vdots$
    \item $( ( (n+1)l_{t-2} + t-3 ) - ((n+1)c + (t-2) ) )  \pmod{(n+1)} \equiv -1$
    \item $( ( (n+1)c + (t-1) ) - ((n+1)c + (t-2) ) )  \pmod{(n+1)} \equiv 1$
    \item $( ( (n+1)c + (t-0) ) - ((n+1)c + (t-2) ) )  \pmod{(n+1)} \equiv 2$\\
     $\vdots$
    \item $( ( (n+1)c + (q) ) - ((n+1)c + (t-2) ) )  \pmod{(n+1)} \equiv q-t+2$
    \item $( ( (n+1)l_{t-1} + (q+1) ) - ((n+1)c + (t-2) ) )  \pmod{(n+1)} \equiv q-t+3$
    \item $( ( (n+1)c + (q+2) ) - ((n+1)c + (t-2) ) )  \pmod{(n+1)} \equiv q-t+4$
    \item $( ( (n+1)c + (q+3) ) - ((n+1)c + (t-2) ) )  \pmod{(n+1)} \equiv q-t+5$\\
     $\vdots$
    \item $( ( (n+1)c + (n) ) - ((n+1)c + (t-2) ) )  \pmod{(n+1)} \equiv n-t+2 \equiv 1-t$
\end{itemize}

The differences covered by $L_{t-1}$ are:

\begin{itemize}
    \item $( ( (n+1)l_{t-1} + (0) ) - ( (n+1)c + (t-1) ) )  \pmod{(n+1)} \equiv 1-t$
    \item $( ( (n+1)l_{t-1} + (1) ) - ( (n+1)c + (t-1) ) )  \pmod{(n+1)} \equiv 2-t$
    \item $( ( (n+1)l_{t-1} + (2) ) - ( (n+1)c + (t-1) ) )  \pmod{(n+1)} \equiv 3-t$\\
     $\vdots$
    \item $( ( (n+1)l_{t-1} + (q) ) - ( (n+1)c + (t-1) ) )  \pmod{(n+1)} \equiv q-t+1$
    \item $( ( (n+1)l_{t-1} + (q+2) ) - ( (n+1)c + (t-1) ) )  \pmod{(n+1)} \equiv q-t+3$
    \item $( ( (n+1)l_{t-1} + (q+3) ) - ( (n+1)c + (t-1) ) )  \pmod{(n+1)} \equiv q-t+4$\\
     $\vdots$
    \item $( ( (n+1)l_{t-1} + (n) ) - ( (n+1)c + (t-1) ) )  \pmod{(n+1)} \equiv n-t+1$
\end{itemize}

Thus, the differences covered by $L_{t-2}$ and $L_{t-1}$ can be recorded as rows of $T_{t-2}^1$ and $T_{t-1}^1$ respectively.\\

Case $t = q+2$: The $n$ forward prime differences in a $L_i$ are as below.

For $i=0, \dots, t-3$, the $n$ forward prime differences in $L_i$ are:
\begin{itemize}
    \item $( ( (n+1)l_{i+1}+(i+1)) - ( (n+1)c+i ) )  \pmod{(n+1)} \equiv 1$
    \item $( ( (n+1)l_{i+1}+(i+2)) - ( (n+1)c+i ) )  \pmod{(n+1)} \equiv 2$
    \item $( ( (n+1)l_{i+1}+(i+3)) - ( (n+1)c+i ) )  \pmod{(n+1)} \equiv 3$\\
     $\vdots$
    \item $( ( (n+1)l_{i+1}+(i+n)) - ( (n+1)c+i ) )  \pmod{(n+1)} \equiv n$
\end{itemize}

So the differences covered by $L_i$ can be recorded as a row of $T_i^1$ with no empty cells.\\

Then, the difference covered by $L_{t-2}$ are:
\begin{itemize}
    \item $( ( (n+1)l_1 + 0 ) - ((n+1)c + (t-2) ) )  \pmod{(n+1)} \equiv 2-t$
    \item $( ( (n+1)l_2 + 1 ) - ((n+1)c + (t-2) ) )  \pmod{(n+1)} \equiv 3-t$
    \item $( ( (n+1)l_2 + 2 ) - ((n+1)c + (t-2) ) )  \pmod{(n+1)} \equiv 4-t$\\
     $\vdots$
    \item $( ( (n+1)l_{t-2} + (t-3) ) - ((n+1)c + (t-2) ) )  \pmod{(n+1)} \equiv -1$
    \item $( ( (n+1)l_{t-1} + (t-1) ) - ((n+1)c + (t-2) ) )  \pmod{(n+1)} \equiv 1$
    \item $( ( (n+1)c+(t) ) - ((n+1)c + (t-2) ) )  \pmod{(n+1)} \equiv 2$
    \item $( ( (n+1)c+(t+1) ) - ((n+1)c + (t-2) ) )  \pmod{(n+1)} \equiv 3$
    \item $( ( (n+1)c+(t+2) ) - ((n+1)c + (t-2) ) )  \pmod{(n+1)} \equiv 4$\\
     $\vdots$
    \item $( ( (n+1)c + (n) ) - ((n+1)c + (t-2) ) )  \pmod{(n+1)} \equiv n-t+2 \equiv 1-t$
\end{itemize}

So the differences covered by $L_i$ can be recorded as a row of $T_{t-2}^1$ with no empty cells.\\

Then, the difference covered by $L_{t-1}$ are:
\begin{itemize}
    \item $( ( (n+1)l_{t-1} + (0) ) - ( (n+1)c + (t-1) ) )  \pmod{(n+1)} \equiv 1-t$
    \item $( ( (n+1)l_{t-1} + (1) ) - ( (n+1)c + (t-1) ) )  \pmod{(n+1)} \equiv 2-t$
    \item $( ( (n+1)l_{t-1} + (2) ) - ( (n+1)c + (t-1) ) )  \pmod{(n+1)} \equiv 3-t$\\
     $\vdots$
    \item $( ( (n+1)l_{t-1} + (q+0) ) - ( (n+1)c + (t-1) ) )  \pmod{(n+1)} \equiv q-t+1 = -1$
    \item $( ( (n+1)l_{t-1} + (q+2) ) - ( (n+1)c + (t-1) ) )  \pmod{(n+1)} \equiv q-t+3 = 1$\\
     $\vdots$
    \item $( ( (n+1)l_{t-1} + (n) ) - ( (n+1)c + (t-1) ) )  \pmod{(n+1)} \equiv n-t+1 = -t$
\end{itemize}

So the differences covered by $L_i$ can be recorded as a row of $T_{t-1}^1$ with no empty cells.\\

Case $t > q+2$, for $i=0,1,\dots,t-2$, the $n$ forward prime differences in $L_i$ are:

\begin{itemize}
    \item $( ( (n+1)l_{i+1}+(i+1)) - ( (n+1)c+i ) )  \pmod{(n+1)} \equiv 1$
    \item $( ( (n+1)l_{i+1}+(i+2)) - ( (n+1)c+i ) )  \pmod{(n+1)} \equiv 2$
    \item $( ( (n+1)l_{i+1}+(i+3)) - ( (n+1)c+i ) )  \pmod{(n+1)} \equiv 3$\\
     $\vdots$
    \item $( ( (n+1)l_{i+1}+(i+n)) - ( (n+1)c+i ) )  \pmod{(n+1)} \equiv n$
\end{itemize}

So the differences covered by $L_i$ can be recorded as a row of $T_i^1$ with no empty cells.

Then, the difference covered by $L_{t-1}$ are:

\begin{itemize}
    \item $( ( (n+1)l_1 + 0 ) - ( (n+1)c+(t-1) ) )  \pmod{(n+1)} \equiv 1-t$
    \item $( ( (n+1)l_2 + 1 ) - ( (n+1)c+(t-1) ) )  \pmod{(n+1)} \equiv 2-t$
    \item $( ( (n+1)l_3 + 2 ) - ( (n+1)c+(t-1) ) )  \pmod{(n+1)} \equiv 3-t$\\
     $\vdots$
    \item $( ( (n+1)l_{t-1} + (t-2) ) - ( (n+1)c+(t-1) ) )  \pmod{(n+1)} \equiv -1$
    \item $( ( (n+1)c+(t+0) ) - ( (n+1)c+(t-1) ) )  \pmod{(n+1)} \equiv 1$
    \item $( ( (n+1)c+(t+1) ) - ( (n+1)c+(t-1) ) )  \pmod{(n+1)} \equiv 2$
    \item $( ( (n+1)c+(t+2) ) - ( (n+1)c+(t-1) ) )  \pmod{(n+1)} \equiv 3$\\
     $\vdots$
    \item $( ( (n+1)c+(n) ) - ( (n+1)c+(t-1) ) )  \pmod{(n+1)} \equiv n-t+1 = 0-t$
\end{itemize}

So the differences covered by $L_i$ can be recorded as a row of $T_{t-1}^1$ with no empty cells.\\

Recall, there are the total $q$ forward or backward prime differences contained in each mixed star, and we introduced symbol $\oplus$ to build mixed stars in Section~\ref{Part I}. For a mixed star $m_i$ from Part I, the difference of a prime edge $l_j'$ is $( ( (n+1)l_{j}'+((i \oplus 1) ) - ( (n+1)c+i ) )  \pmod{(n+1)}$ for some $j \in 1,2,\dots,q$.

Now suppose $m_i$ contains backward leaves. Consider any backward leaf $l_j'$. If $l_j'$ had been a forward prime, it would have had the following difference:
\[ (n+1)l_j' + (i+j) - ((n+1)c+i) \equiv j \pmod{n+1}.  \]
Therefore, it would have been recorded in column $j$ of table $T_i^1$. Instead, because it is a backward leaf, it gets recorded in table $T_{(i-j)\pmod{n+1}}^1$. However, its difference is:
\begin{align*}
    v - ((n+1)l_j' + (i-j) - ((n+1)c+i)) &\equiv v+j \pmod{n+1} \\
    &\equiv j \pmod{n+1}.
\end{align*}
Thus the difference for $l_j'$ still gets recorded in column $j$. Therefore, every table $T_i^1$ contains a row corresponding to the mixed star $m_i$ with $q$ entries given to the $q$ prime differences and $n-q$ empty cells. Furthermore, the nonempty cells occur in columns $1,2,\dots,q$ and the $n-q$ empty cells occur in column $q+1,q+2,\dots,n$. Hence there is a balanced star array for each $V_i$.

An example of the six balanced star arrays for $V_0,V_1,\dots,V_5$ on $v=162$ vertices is given in Figure~\ref{BSA v162}. These arrays were built to illustrate the differences used in the stars from the Part I factors that are described in Figure~\ref{v162}.

\end{proof}

\begin{figure}[!htb]
    \footnotesize
    \begin{minipage}{.5\linewidth}
        \centering  
        \begin{tabular}{|c|ccccc|}
        \hline
         &  &  & $T_0$ &  & \\
         \hline
        $T_0^1$ & 37 & 32 & 21 & 16 & 11 \\
         & 49 & 50 & 51 & 52 & 53\\
         & 67 & 80 & * & * & * \\
         \hline
        $T_0^2$ & 1 & 2 & 3 & 4 & 5 \\
         & 7 & 8 & 9 & 10 & 17 \\
         & 13 & 14 & 15 & 22 & 23\\
         & 19 & 20 & 27 & 28 & 29 \\
         & 25 & 26 & 33 & 34 & 35  \\
         & 31 & 38 & 39 & 40 & 41 \\
         & 43 & 44 & 45 & 46 & 47 \\
         & 55 & 56 & 57 & 58 & 59 \\
         & 61 & 62 & 63 & 64 & 65 \\
         & 73 & 68 & 69 & 70 & 71 \\
         & 79 & 74 & 75 & 76 & 77 \\
         \hline
       \end{tabular}
    \end{minipage}
    \begin{minipage}{.5\linewidth}
    \begin{tabular}{|c|ccccc|}
        \hline
         &  &  & $T_1$ &  & \\
         \hline
        $T_1^1$ & 37 & 32 & 21 & 16 & 5 \\
         & 1 & 56 & 3 & 4 & 47 \\
         & 67 & 80 & * & * & * \\
         \hline
        $T_1^2$ & 7 & 2 & 9 & 10 & 11 \\
         & 13 & 8 & 15 & 22 & 17\\
         & 19 & 14 & 27 & 28 & 23 \\
         & 25 & 20 & 33 & 34 & 29  \\
         & 31 & 26 & 39 & 40 & 35 \\
         & 43 & 38 & 45 & 46 & 41 \\ 
         & 49 & 44 & 51 & 52 & 53 \\
         & 55 & 50 & 57 & 58 & 59 \\
         & 61 & 62 & 63 & 64 & 65 \\
         & 73 & 68 & 69 & 70 & 71 \\
         & 79 & 74 & 75 & 76 & 77 \\
        \hline
    \end{tabular}
    \end{minipage}
\vskip 0.2in

    \begin{minipage}{.5\linewidth}
        \centering  
        \begin{tabular}{|c|ccccc|}
        \hline
         &  &  & $T_2$ &  & \\
         \hline
        $T_2^1$ & 37 & 32 & 21 & 10 & 5 \\
         & 67 & 80 & * & * & * \\
         \hline
        $T_2^2$ & 1 & 2 & 3 & 4 & 11 \\
         & 7 & 8 & 9 & 16 & 17 \\
         & 13 & 14 & 15 & 22 & 23\\
         & 19 & 20 & 27 & 28 & 29 \\
         & 25 & 26 & 33 & 34 & 35  \\
         & 31 & 38 & 39 & 40 & 41 \\
         & 43 & 44 & 45 & 46 & 47 \\ 
         & 49 & 50 & 51 & 52 & 53 \\
         & 55 & 56 & 57 & 58 & 59 \\
         & 61 & 62 & 63 & 64 & 65 \\
         & 73 & 68 & 69 & 70 & 71 \\
         & 79 & 74 & 75 & 76 & 77 \\
         \hline
       \end{tabular}
    \end{minipage}
    \begin{minipage}{.5\linewidth}
    \begin{tabular}{|c|ccccc|}
        \hline
         &  &  & $T_3$ &  & \\
         \hline
        $T_3^1$ & 37 & 32 & 15 & 10 & 5 \\
         & 67 & 80 & * & * & * \\
         \hline
        $T_3^2$ & 1 & 2 & 3 & 4 & 11 \\
         & 7 & 8 & 9 & 16 & 17 \\
         & 13 & 14 & 21 & 22 & 23\\
         & 19 & 20 & 27 & 28 & 29 \\
         & 25 & 26 & 33 & 34 & 35  \\
         & 31 & 38 & 39 & 40 & 41 \\
         & 43 & 44 & 45 & 46 & 47 \\ 
         & 49 & 50 & 51 & 52 & 53 \\
         & 55 & 56 & 57 & 58 & 59 \\
         & 61 & 62 & 63 & 64 & 65 \\
         & 73 & 68 & 69 & 70 & 71 \\
         & 79 & 74 & 75 & 76 & 77 \\
        \hline
    \end{tabular}
    \end{minipage}
\vskip 0.2in

    \begin{minipage}{.5\linewidth}
        \centering  
        \begin{tabular}{|c|ccccc|}
        \hline
         &  &  & $T_4$ &  & \\
         \hline
        $T_4^1$ & 37 & 26 & 15 & 10 & 5 \\
         & 55 & 56 & 51 & 52 & 53\\
         & 67 & 74 & * & * & * \\
         \hline
        $T_4^2$ & 1 & 2 & 3 & 4 & 11 \\
         & 7 & 8 & 9 & 16 & 17 \\
         & 13 & 14 & 21 & 22 & 23\\
         & 19 & 20 & 27 & 28 & 29 \\
         & 25 & 32 & 33 & 34 & 35  \\
         & 31 & 38 & 39 & 40 & 41 \\
         & 43 & 44 & 45 & 46 & 47 \\ 
         & 49 & 50 & 57 & 58 & 59 \\
         & 61 & 62 & 63 & 64 & 65 \\
         & 73 & 68 & 69 & 70 & 71 \\
         & 79 & 80 & 75 & 76 & 77 \\
         \hline
       \end{tabular}
    \end{minipage}
    \begin{minipage}{.5\linewidth}
    \begin{tabular}{|c|ccccc|}
        \hline
         &  &  & $T_5$ &  & \\
         \hline
        $T_5^1$ & 31 & 26 & 15 & 10 & 5 \\
         & 61 & 74 & * & * & * \\
         \hline
        $T_5^2$ & 1 & 2 & 3 & 4 & 11 \\
         & 7 & 8 & 9 & 16 & 17 \\
         & 13 & 14 & 21 & 22 & 23\\
         & 19 & 20 & 27 & 28 & 29 \\
         & 25 & 32 & 33 & 34 & 35  \\
         & 37 & 38 & 39 & 40 & 41 \\
         & 43 & 44 & 45 & 46 & 47 \\ 
         & 49 & 50 & 51 & 52 & 53 \\
         & 55 & 56 & 57 & 58 & 59 \\
         & 67 & 62 & 63 & 64 & 65 \\
         & 73 & 68 & 69 & 70 & 71 \\
         & 79 & 80 & 75 & 76 & 77 \\
        \hline
    \end{tabular}
    \end{minipage}

    \caption{Balanced star array for $v=162$}
    \label{BSA v162}
\end{figure}

\begin{lemma} 
\label{BSAe}
There is a balanced star array for each $V_{i}$, $i \in \mathbb{Z}_{n+1}$ when $v=2n(n+1)k+2(n+1)$ with $k \geq 1$.
\end{lemma}

\begin{proof}
  
First, we note that the case when $k=0$ is given by Lemma~\ref{result0}.\\

If $v=2n(n+1)k+2(n+1)$, observe that $T$ has $nk+1$ rows, and thus, there are no empty cells in any row. For $T=T^1 \cup T^2$ to be a balanced star array, we must prove that each subarray $T^1$ and $T^2$ contains no empty cells.\\

For each $i \in \{0,1,\dots,n\}$, let $T_i=T_i^1 \cup T_i^2$ be the array for $V_i$ recording the differences used in the Part I factors from Lemma~\ref{Part I}, when $k'$ is even. We will prove that each $T_i$ is balanced. \\

For any prime star, $p_i$, the $n$ forward prime differences are as follows:
\begin{itemize}
    \item $(((n+1)l_1 + n + i ) - ((n+1)c+i) ) \pmod{n+1)} \equiv n$
    \item $(((n+1)l_1 + (n-1) + i ) - ((n+1)c+i) ) \pmod{n+1)} \equiv n-1$
    \item $(((n+1)l_1 + (n-2) + i ) - ((n+1)c+i) ) \pmod{n+1)} \equiv n-2$
    \\ $\vdots$ 
    \item $(((n+1)l_1 + 2 + i ) - ((n+1)c+i) ) \pmod{n+1)} \equiv 2$
    \item $(((n+1)l_1 + 1 + i ) - ((n+1)c+i) ) \pmod{n+1)} \equiv 1$
\end{itemize}

Thus, each prime star's differences can be recorded as a row of $T^1_i$ with no empty cells.\\

Then, for any little star, $L_i$, the $n$ forward prime differences are as follows.

For $i=0, \dots, t-2$, the forward prime differences in $L_i$ are:
\begin{itemize}
    \item $( ( (n+1)l_{i+1}+(i+1)) - ( (n+1)c+i ) )  \pmod{(n+1)} \equiv 1$
    \item $( ( (n+1)l_{i+1}+(i+2)) - ( (n+1)c+i ) )  \pmod{(n+1)} \equiv 2$
    \item $( ( (n+1)l_{i+1}+(i+3)) - ( (n+1)c+i ) )  \pmod{(n+1)} \equiv 3$
    \\ $\vdots$
    \item $( ( (n+1)l_{i+1}+(i+n)) - ( (n+1)c+i ) )  \pmod{(n+1)} \equiv n$
\end{itemize}

Thus, each little star's differences can be recorded as a row of $T^1_i$ for $i=0, \dots, t-2$ with no empty cells.\\

Then, the forward prime differences in $L_{t-1}$ is:
\begin{itemize}
    \item $( ( (n+1)l_1 + 0 ) - ( (n+1)c+(t-1) ) )  \pmod{(n+1)} \equiv 1-t$
    \item $( ( (n+1)l_2 + 1 ) - ( (n+1)c+(t-1) ) )  \pmod{(n+1)} \equiv 2-t$
    \item $( ( (n+1)l_3 + 2 ) - ( (n+1)c+(t-1) ) )  \pmod{(n+1)} \equiv 3-t$
    \\ $\vdots$
    \item $( ( (n+1)l_{t-1} + (t-2) ) - ( (n+1)c+(t-1) ) )  \pmod{(n+1)} \equiv -1$
    \item $( ( (n+1)c+(t+0) ) - ( (n+1)c+(t-1) ) )  \pmod{(n+1)} \equiv 1$
    \item $( ( (n+1)c+(t+1) ) - ( (n+1)c+(t-1) ) )  \pmod{(n+1)} \equiv 2$
    \item $( ( (n+1)c+(t+2) ) - ( (n+1)c+(t-1) ) )  \pmod{(n+1)} \equiv 3$
    \\ $\vdots$
    \item $( ( (n+1)c+(n) ) - ( (n+1)c+(t-1) ) )  \pmod{(n+1)} \equiv n-t+1 = 0-t$
\end{itemize}

Thus, a little star $L_{t-1}$'s differences can be recorded as a row of $T^1_{t-1}$ with no empty cells.\\

Therefore, we've shown that the rows in each $T^1_i$ contain no empty cells. Hence, there is a balanced star array for each $V_i$.

\end{proof}


\section{Results}
We begin with the case of $k=0$, in Lemma~\ref{result0}. Then, we will provide the general case, $k\geq1$, in Lemma~\ref{result1}.

\begin{lemma}
\label{result0}
Let $v=2(n+1)$. A $(K_2,K_{1,n})$-$URD(v;1,s)$ exists for all odd $n\geq3$.
\end{lemma}

\begin{proof}
       Let $V=\{0,1,\dots,2n+1 \}$ be the vertex set. We partition the vertex set into $n+1$ parts by taking $V$ to be $V = \bigcup_{i=0}^n V_i$ where $V_i=\{\nu \in V | \nu\equiv i \pmod{(n+1)} \}$. Let $T_i = \{1,2,\dots,n+1\}$ be the set of forward differences of the edge $\{c,v\}$ such that $c \in V_i$ is the center of a star. First, we construct a $1$-factor $I$ as follows:
\begin{align*}
    I = \{(i,i+(n+1)) : i \in \{0,1,2,\dots, n \} \}.
\end{align*}
       Then, we construct $n+1$ $n$-star factors $F_i$ for all $i=0,1,\dots,n$ as follows: 
\begin{align*}
    F_i = \{ &(0+i; 1+i,2+i,\dots,n+i),\\
    &(n+1+i;n+2+i,n+3+i,\dots,2n+1+i) \}.
\end{align*}

     We must prove that these factors cover all differences in each $T_i$. Note that $I$ covers the difference $n+1$ for each $T_i$. Then for each $i$,  $F_i$ covers the differences $\{1,2,\dots, n\} \in T_i$. Hence, $I$ and $F_i$ cover all differences in $T_i$, and thus, a $(K_2,K_{1,n})$-$URD(v;1,s)$ exists.

\end{proof}

\begin{lemma}
\label{result1}
Let $v=n(n+1)k'+2(n+1)$ for any non-negative integer $k'$. There exists a $(K_2,K_{1.n})-URD(v;1,s)$ for all odd $n\geq3$.
\end{lemma}

\begin{proof}
When $k'=0$, we have $v=2(n+1)$, and the result follows Lemma~\ref{result0}.

If $k' \geq 1$, there exists an {\em almost n-star factor} on $K_{\frac{v}{(n+1)}}$ by Lemma~\ref{todd1} $\sim$ Lemma~\ref{teven2}. Then, there exists $v$ $n$-star factors on $v$ vertices by Lemma~\ref{Part I}. Let the one $1$-factor $I$ be $I= \{ (u,v) : D(u,v) = \frac{v}{2} \}$. By Lemma~\ref{BSAo} and Lemma~\ref{BSAe}, there exists a balanced star array for each $V_i$ with $i \in \mathbb{Z}_{(n+1)}$. Thus, by Lemma~\ref{Part II}, the remaining edges of $K_v$ can be decomposed into $n$-star factors. 

\end{proof}

Lemma~\ref{result0} and ~\ref{result1} show that if $v$ satisfies our necessary condition, then a $(K_2,K_{1,n})$-$URD(v;1,s)$ exists for any odd $n$. Thus, our main result is proven.

\begin{thm}
    Let $n>1$ be an odd integer. A $(K_2,K_{1,n})$-$URD(v;1,s)$ exists if and only if $v \equiv 2(n+1) \pmod{n(n+1)}$.

\end{thm}

\end{document}